\documentclass[12pt]{article}
\usepackage{authblk}
\usepackage{amsmath}
\usepackage{amssymb}
\usepackage{amsthm}
\usepackage{enumerate}
\usepackage{graphicx}
\usepackage{epstopdf}
\usepackage[colorlinks]{hyperref}
\usepackage{subfigure}
\usepackage{enumitem}
\usepackage[margin=1in]{geometry}
\usepackage{varwidth}
\usepackage[title]{appendix}

\usepackage{pgfplots, pgfplotstable, booktabs, colortbl, siunitx, array}

\usepackage{algorithm,algpseudocode}

\newcommand{\pluseq}{\mathrel{+}=}

\usepackage[style=numeric-comp,firstinits=true,url=false,maxbibnames=5,backend=bibtex]{biblatex}
\bibliography{references.bib}

\newtheorem{theorem}{Theorem}[section]
\newtheorem{lemma}[theorem]{Lemma}
\newtheorem{proposition}[theorem]{Proposition}

\usepackage{mathtools}

\usepackage{tikz}
\usetikzlibrary{positioning,arrows}

\title{Interpolation on Symmetric Spaces via the Generalized Polar Decomposition}
\author[1]{Evan S. Gawlik\thanks{\textit{E-mail address:} \texttt{egawlik@ucsd.edu} (Corresponding author)}}
\author[1]{Melvin Leok\thanks{\textit{E-mail address:} \texttt{mleok@math.ucsd.edu}}}
\affil{Department of Mathematics, University of California, San Diego.}

\begin{document}

\date{}

\maketitle

\begin{abstract}
We construct interpolation operators for functions taking values in a symmetric space -- a smooth manifold with an inversion symmetry about every point.  Key to our construction is the observation that every symmetric space can be realized as a homogeneous space whose cosets have canonical representatives by virtue of the generalized polar decomposition -- a generalization of the well-known factorization of a real nonsingular matrix into the product of a symmetric positive-definite matrix times an orthogonal matrix.  By interpolating these canonical coset representatives, we derive a family of structure-preserving interpolation operators for symmetric space-valued functions.  As applications, we construct interpolation operators for the space of Lorentzian metrics, the space of symmetric positive-definite matrices, and the Grassmannian.  In the case of Lorentzian metrics, our interpolation operators provide a family of finite elements for numerical relativity that are frame-invariant and have signature which is guaranteed to be Lorentzian pointwise.  We illustrate their potential utility by interpolating the Schwarzschild metric numerically.
\end{abstract}

\section{Introduction}

Manifold-valued data and manifold-valued functions play an important role in a wide variety of applications, including mechanics~\cite{sander2010geodesic,demoures2015discrete,hall2014lie}, computer vision and graphics~\cite{hong2014geodesic,turaga2011statistical,gallivan2003efficient,chang2012feature,de2014discrete,jiang2015frame}, medical imaging~\cite{arsigny2006log}, and numerical relativity~\cite{arnold2000numerical}.  
 By their very nature, such applications demand that care be taken when performing computations that would otherwise be routine, such as averaging, interpolation, extrapolation, and the numerical solution of differential equations.  This paper constructs interpolation and averaging operators for functions taking values in a \emph{symmetric space} -- a smooth manifold with an inversion symmetry about every point.  Key to our construction is the observation that every symmetric space can be realized as a homogeneous space whose cosets have canonical representatives by virtue of the \emph{generalized polar decomposition} -- a generalization of the well-known factorization of a real nonsingular matrix into the product of a symmetric positive-definite matrix times an orthogonal matrix.  By interpolating these canonical coset representatives, we derive a family of structure-preserving interpolation operators for symmetric space-valued functions.

Our motivation for constructing such operators is best illustrated by example.  Among the most interesting scenarios in which symmetric space-valued functions play a role is numerical relativity. There, 
the dependent variable in Einstein's equations~--~the metric tensor~--~is a function taking values in the space $\mathcal{L}$ of Lorentzian metrics: bilinear, symmetric, nondegenerate $2$-tensors with signature $(3,1)$.  This space is neither a vector space nor a convex set.   Rather, it has the structure of a symmetric space.  As a consequence, the outputs of basic arithmetic operations on Lorentzian metrics such as averaging, interpolation, and extrapolation need not remain in $\mathcal{L}$.  This is undesirable for several reasons.  If the metric tensor field is to be discretized with finite elements, then a naive approach in which the components of the metric are discretized with piecewise polynomials may fail to produce a metric field with signature $(3,1)$ at all points in spacetime.  Perhaps an even more problematic possibility is that a numerical time integrator used to advance the metric forward in time (e.g., in a $3+1$ formulation of Einstein's equations) might produce metrics with invalid signature.  One of the aims of the present paper is to avert these potential dangers altogether by constructing a structure-preserving interpolation operator for Lorentzian metrics.  As will be shown, the interpolation operator we derive not only produces interpolants that everywhere belong to $\mathcal{L}$, but it is also frame-invariant: the interpolation operator we derive commutes with the action of the indefinite orthogonal group $O(1,3)$ on $\mathcal{L}$.  Furthermore, our interpolation operator commutes with inversion and interpolates the determinant of the metric tensor in a monotonic manner.

A more subtle example is the space $SPD(n)$ of symmetric positive definite $n \times n$ matrices.  This space forms a convex cone, so arithmetic averaging and linear interpolation trivially produce $SPD(n)$-valued results.  
Nevertheless, these operations fail to preserve other structures that are important in some applications. 
For instance, arithmetic averaging does not commute with matrix inversion, and the determinant of the arithmetic average need not be less than or equal to the maximum of the determinants of the data.  This may remedied by considering instead  the Riemannian mean (also known as the Karcher mean) of symmetric positive-definite matrices with respect to the canonical left-invariant Riemannian metric on $SPD(n)$~\cite{moakher2005differential,bhatia2013riemannian,karcher1977riemannian}.  The Riemannian mean cannot, in general, be expressed in closed form, but it can be computed iteratively and possesses a number of structure-preserving properties; see~\cite{bhatia2013riemannian} for details.  A less computationally expensive alternative, introduced by Arsigny and co-authors~\cite{arsigny2007geometric}, is to compute the mean of symmetric positive-definite matrices with respect to a log-Euclidean metric on $SPD(n)$.  The resulting averaging operator commutes with matrix inversion, prevents overestimation of the determinant, and commutes with similarity transformations that consist of an isometry plus scaling.  Both of these constructions turn out to be a special cases of the general theory presented in this paper.  In our derivation of the log-Euclidean mean, we give a clear geometric explanation of the vector space structure with which Arsigny and co-authors~\cite{arsigny2007geometric} endow $SPD(n)$ in their derivation, which turns out to be nothing more than a correspondence between a symmetric space ($SPD(n)$) and a Lie triple system~\cite{helgason1979differential}.

Another symmetric space which we address in this paper is the Grassmannian $Gr(p,n)$, which consists of all $p$-dimensional linear subspaces of $\mathbb{R}^n$.  Interpolation on the Grassmannian is a task of importance in a variety of contexts, including reduced-order modeling~\cite{amsallem2008interpolation,vetrano2012assessment} and computer vision~\cite{hong2014geodesic,turaga2011statistical,gallivan2003efficient,chang2012feature}.  Not surprisingly, this task has received much attention in the literature; see~\cite{begelfor2006affine,absil2004riemannian} and the references therein.  Our constructions in this paper recover some of the well-known interpolation schemes on the Grassmannian, including those that appear in~\cite{amsallem2008interpolation,chang2012feature,begelfor2006affine}.

There are connections between the present work and geodesic finite elements~\cite{grohs2015optimal,grohs2013quasi,sander2012geodesic,sander2015geodesic}, a family of conforming finite elements for functions taking values in a Riemannian manifold $M$.  In fact, we recover such elements as a special case in the theory below; see Section~\ref{sec:interp_generalizations}.  Since their evaluation amounts to the computation of a weighted Riemannian mean, geodesic finite elements and their derivatives can sometimes be expensive to compute.  One of the messages we hope to convey is that when $M$ is a symmetric space, this additional structure enables the construction of alternative interpolants that are less expensive to compute but still possess many of the desirable features of geodesic finite elements.

Our use of the generalized polar decomposition in this paper is inspired by a stream of research~\cite{munthe2001generalized,munthe2014symmetric,iserles2005efficient} that has, in recent years, cast a spotlight on the generalized polar decomposition's role in numerical analysis.  Much of our exposition and notation parallels that which appears in those papers, and we encourage the reader to look there for further insight.

\paragraph{Organization.}
This paper is organized as follows.  We begin in Section~\ref{sec:symspace} by reviewing symmetric spaces, Lie triple systems, and the generalized polar decomposition.  Then, in Section~\ref{sec:interp}, we exploit a correspondence between symmetric spaces and Lie triple systems to construct interpolation operators on symmetric spaces.  Finally, in Section~\ref{sec:applications}, we specialize these interpolation operators to three examples of symmetric spaces: the space of symmetric positive-definite matrices, the space of Lorentzian metrics, and the Grassmannian.  In the case of Lorentzian metrics, we illustrate the potential utility of these interpolation operators by interpolating the Schwarzschild metric numerically.

\section{Symmetric Spaces and the Generalized Polar Decomposition} \label{sec:symspace}

In this section, we review symmetric spaces, Lie triple systems, and the generalized polar decomposition.  We describe a well-known correspondence between symmetric spaces and Lie triple systems that will serve in Section~\ref{sec:interp} as a foundation for interpolating functions which take values in a symmetric space.  For further background material, we refer the reader to~\cite{helgason1979differential,munthe2001generalized,munthe2014symmetric}.

\subsection{Notation and Definitions} \label{sec:notation}

Let $G$ be a Lie group and let $\sigma : G \rightarrow G$ be an involutive automorphism.  That is, $\sigma \neq \mathrm{id}.$ is a bijection satisfying $\sigma(\sigma(g))=g$ and $\sigma(gh)=\sigma(g)\sigma(h)$ for every $g,h \in G$.  Denote by $G^\sigma$ the subgroup of $G$ consisting of fixed points of $\sigma$:
\[
G^\sigma = \{g \in G \mid \sigma(g)=g\}.
\]
Suppose that $G$ acts transitively on a smooth manifold $\mathcal{S}$ with a distinguished element $\eta \in \mathcal{S}$ whose stabilizer coincides with $G^\sigma$.  In other words,
\[
g \cdot \eta = \eta \iff \sigma(g)=g
\]
where $g \cdot u$ denotes the action of $g \in G$ on an element $u \in \mathcal{S}$.  Then there is a bijective correspondence between elements of the homogeneous space $G/G^\sigma$ and elements of $\mathcal{S}$.  On the other hand, the cosets in $G/G^\sigma$ have canonical representatives by virtue of the \emph{generalized polar decomposition}~\cite{munthe2001generalized,munthe2014symmetric}.  This decomposition states that any $g \in G$ sufficiently close to the identity $e \in G$ can be written as a product
\begin{equation}\label{gpd}
g = pk, \quad p \in G_\sigma, \, k \in G^\sigma,
\end{equation}
where 
\[
G_\sigma = \{g \in G \mid \sigma(g) = g^{-1} \}.
\]
Moreover, this decomposition is locally unique~\cite[Theorem 3.1]{munthe2001generalized}.  As a consequence, there is a bijection between a neighborhood of the identity $e \in G_\sigma$ and a neighborhood of the coset $[e] \in G/G^\sigma$.  The space $G_\sigma$ -- which, unlike $G^\sigma$, is not a subgroup of $G$ -- is a symmetric space which is closed under a non-associative symmetric product $g \cdot h = gh^{-1} g$.  Its tangent space at the identity is the space 
\[
\mathfrak{p} = \{Z \in \mathfrak{g} \mid d\sigma(Z) = -Z\}.
\]
Here, $\mathfrak{g}$ denotes the Lie algebra of $G$, and $d\sigma : \mathfrak{g} \rightarrow \mathfrak{g}$ denotes the differential of $\sigma$ at $e$, which can be expressed in terms of the Lie group exponential map $\exp : \mathfrak{g} \rightarrow G$ via
\[
d\sigma(Z) = \left.\frac{d}{dt}\right|_{t=0} \sigma(\exp(tZ)).
\]
The space $\mathfrak{p}$, which is not a Lie subalgebra of $\mathfrak{g}$, has the structure of Lie triple system: it is a vector space closed under the double commutator $[ \cdot, [\cdot,\cdot]]$.  In contrast, the space
\[
\mathfrak{k} = \{Z \in \mathfrak{g} \mid d\sigma(Z) = Z\}
\]
is a subalgebra of $\mathfrak{g}$, as it is closed under the commutator $[\cdot,\cdot]$.  This subalgebra is none other than the Lie algebra of $G^\sigma$.  The generalized polar decomposition~(\ref{gpd}) has a manifestation at the Lie algebra level called the \emph{Cartan decomposition}, which decomposes $\mathfrak{g}$ as a direct sum
\begin{equation} \label{cartan}
\mathfrak{g} = \mathfrak{p} \oplus \mathfrak{k}.
\end{equation}
All of these observations lead to the conclusion that the following diagram commutes:

\begin{center}
\begin{tikzpicture}
\node (G) {$G$};
\node[below=1.5cm of G] (quotient) {$G/G^\sigma$}; 
\node[left=1.5cm of quotient] (Gsub) {$G_\sigma$};
\node[right=1.5cm of quotient] (X) {$\mathcal{S}$};
\node[left=1.5cm of Gsub] (p) {$\mathfrak{p}$};
\node[above=1.5cm of p] (g) {$\mathfrak{g}=\mathfrak{p}\oplus\mathfrak{k}$};
\node[above=1.5cm of g] (k) {$\mathfrak{k}$};
\node[right=1.5cm of k] (Gsup) {$G^\sigma$};
\draw[-latex]
(G) edge node (Gtoquotient) [left]{$\pi$} (quotient)
(G) edge node (GtoX) [above right]{$\varphi$} (X)
(quotient) edge node (quotienttoX) [below]{$\bar{\varphi}$} (X)
(Gsub) edge node (Gsubtoquotient) [below]{$\psi$} (quotient)
(Gsub) edge node (GsubtoG) [above left]{$\iota$} (G)
(p) edge node (ptoGsub) [below]{$\exp$} (Gsub)
(p) edge node (ptog) [left]{$\iota$} (g)
(g) edge node (gtoG) [above]{$\exp$} (G)
(k) edge node (ktog) [left]{$\iota$} (g)
(k) edge node (ktoGsup) [above]{$\mathrm{exp}$} (Gsup)
(Gsup) edge node (GsuptoG) [above right]{$\iota$} (G);
\end{tikzpicture}
\end{center}

In this diagram, we have used the letter $\iota$ to denote the canonical inclusion, $\pi : G \rightarrow G/G^\sigma$ the canonical projection, and $\varphi : G \rightarrow \mathcal{S}$ the map $\varphi(g) = g \cdot \eta$.  The maps $\psi$ and $\bar{\varphi}$ are defined by the condition that the diagram be commutative.

\subsection{Correspondence between Symmetric Spaces and Lie Triple Systems} \label{sec:correspondence}

An important feature of the diagram above is that the maps along its bottom row -- when restricted to suitable neighborhoods of the neutral elements $0 \in \mathfrak{p}$, $e \in G_\sigma$, $[e] \in G/G^\sigma$, and the distinguished element $\eta \in \mathcal{S}$ -- are diffeomorphisms~\cite{helgason1979differential}.  
In particular, the composition
\begin{equation} \label{bijection}
F = \bar{\varphi} \circ \psi \circ \exp
\end{equation}
(or, equivalently, $F = \varphi \circ \iota \circ \exp$) provides a diffeomorphism from a neighborhood of $0 \in \mathfrak{p}$ to a neighborhood of  $\eta \in \mathcal{S}$, given by
\[
F(P) = \exp(P) \cdot \eta
\]  
for $P \in \mathfrak{p}$.  The space $\mathfrak{p}$, being a vector space, offers a convenient space to perform computations (such as averaging, interpolation, extrapolation, and the numerical solution of differential equations) that might otherwise be unwieldy on the space $\mathcal{S}$.  This is analogous to the situation that arises when working with the Lie group $G$.  Often, computations on $G$ are more easily performed by mapping elements of $G$ to the Lie algebra $\mathfrak{g}$ via the inverse of the exponential map (or an approximation thereof), performing computations in $\mathfrak{g}$, and mapping the result back to $G$ via the exponential map (or an approximation thereof).

We remark that the analogy just drawn between computing on Lie groups and computing on symmetric spaces is in fact more than a mere resemblance; the latter situation directly generalizes the former.  Indeed, any Lie group $G$ can be realized as a symmetric space by considering the action of $G \times G$ on $G$ given by $(g,h) \cdot k = gkh^{-1}$.  The stabilizer of $e \in G$ is the diagonal of $G \times G$, which is precisely the subgroup fixed by the involution $\sigma(g,h) = (h,g)$.  In this setting, one finds that the map~(\ref{bijection}) takes $(X,-X) \in \mathfrak{g} \times \mathfrak{g}$ to $\exp(2X) \in G$.  This shows that, up to a trivial modification, the map~(\ref{bijection}) reduces to the Lie group exponential map if $\mathcal{S}$ happens to be a Lie group.

An additional feature of the map~(\ref{bijection}) is its equivariance with respect to the action of $G^\sigma$ on $\mathcal{S}$ and $\mathfrak{p}$.  Specifically, for $g \in G$, let $\mathrm{Ad}_g : \mathfrak{g} \rightarrow \mathfrak{g}$ denote the adjoint action of $G$ on $\mathfrak{g}$:
\[
\mathrm{Ad}_g Z = \left.\frac{d}{dt}\right|_{t=0} g \exp(tZ) g^{-1}.
\]
In a slight abuse of notation, we will write
\[
\mathrm{Ad}_g Z = g Z g^{-1}
\]
in this paper, bearing in mind that the above equality holds in the sense of matrix multiplication for any matrix group.
The following lemma shows that $F \circ \mathrm{Ad}_g \big|_{\mathfrak{p}}  = g \cdot F$ for every $g \in G^\sigma$.  Note that this statement makes implicit use of the (easily verifiable) fact that $\mathrm{Ad}_g$ leaves $\mathfrak{p}$ invariant when $g \in G^\sigma$; that is, $gPg^{-1} \in \mathfrak{p}$ for every $g \in G^\sigma$ and every $P \in \mathfrak{p}$.

\begin{lemma} \label{lemma:equivariance}
For every $P \in \mathfrak{p}$ and every $g \in G^\sigma$,
\[
g \cdot F(P) = F(gPg^{-1}).
\]
\end{lemma}
\begin{proof}
Note that $g \in G^\sigma$ implies $g^{-1} \in G^\sigma$, so $g^{-1} \cdot \eta = \eta$.  Hence, since the adjoint action commutes with exponentiation,
\begin{align*}
F(gPg^{-1})
&= \exp(gPg^{-1}) \cdot \eta \\
&= g\exp(P) g^{-1} \cdot \eta \\
&= g\exp(P) \cdot \eta \\
&= g \cdot F(P).
\end{align*}
\end{proof}

We finish this section by remarking that if $\mathcal{S}$ is a Riemannian manifold, then $\sigma$ induces a family of geodesic symmetries on $\mathcal{S}$ as follows.  Define $s_\eta : \mathcal{S} \rightarrow \mathcal{S}$ by setting
\[
s_\eta(g \cdot \eta) = \sigma(g) \cdot \eta.
\] 
for each $g \in G$.  Note that $s_\eta$ is a well-defined isometry that fixes $\eta$ and has differential equal to minus the identity.
Furthermore, by definition,  the following diagram commutes:
\begin{center}
\begin{tikzpicture}
\node (p) {$\mathfrak{p}$};
\node[below=1.5cm of G] (p2) {$\mathfrak{p}$}; 
\node[right=1.5cm of p] (G) {$G$};
\node[right=1.5cm of p2] (G2) {$G$};
\node[right=1.5cm of G] (S) {$\mathcal{S}$};
\node[right=1.5cm of G2] (S2) {$\mathcal{S}$};
\draw[-latex]
(p) edge node (ptoG) [above]{$\exp$} (G)
(p2) edge node (p2toG2) [below]{$\exp$} (G2)
(G) edge node (GtoS) [above]{$\varphi$} (S)
(G2) edge node (G2toS2) [below]{$\varphi$} (S2)
(p) edge node (ptop2) [left]{$d\sigma$} (p2)
(G) edge node (GtoG2) [left]{$\sigma$} (G2)
(S) edge node (StoS2) [right]{$s_\eta$} (S2);
\end{tikzpicture}
\end{center}
Written another way,
\begin{equation} \label{tauF}
s_\eta(F(P)) = F(-P)
\end{equation}
for every $P \in \mathfrak{p}$.

In a similar manner, a geodesic symmetry at each point $h \cdot \eta \in \mathcal{S}$ can be defined via
\[
s_{h \cdot \eta} (g \cdot \eta) = h \cdot s_\eta(h^{-1} g \cdot \eta) = h \sigma(h^{-1}g) \cdot \eta.
\]
For every such $h \in G$, the map $s_ {h \cdot \eta}$ is an isometry, showing that $\mathcal{S}$ is a symmetric space.

\subsection{Generalizations}  \label{sec:correspondence_generalizations}

The construction above can be generalized by replacing the exponential map in~(\ref{bijection}) with a different local diffeomorphism.  One example is given by fixing an element $\bar{g} \in G$ and replacing $\exp : \mathfrak{p} \rightarrow G_\sigma$ in~(\ref{bijection}) with the map
\begin{equation} \label{generalexp}
P \mapsto \psi^{-1}\left([\bar{g} \exp(P)]\right).
\end{equation}
The output of this map is nothing more than the factor $p$ in the generalized polar decomposition $\bar{g}\exp(P)=pk$, $p \in G_\sigma$, $k \in G^\sigma$. The map~(\ref{bijection}) then becomes 
\begin{equation} \label{generalbijection}
F_{\bar{g}}(P) = \bar{g}\exp(P) \cdot \eta.
\end{equation}
This generalization of~(\ref{bijection}) has the property that it provides a diffeomorphism between a neighborhood of $0 \in \mathfrak{p}$ and a neighborhood of $\bar{g} \cdot \eta \in \mathcal{S}$ rather than $\eta$.  Note that when $\bar{g}=e$ (the identity element), this map coincides with~(\ref{bijection}).  A calculation similar to the proof of Lemma~\ref{lemma:equivariance} shows that the map $F_{\bar{g}}$ is $G^\sigma$-equivariant, in the sense that
\begin{equation} \label{generalequivariance}
F_{h\bar{g}h^{-1}}(hPh^{-1}) = h \cdot F_{\bar{g}}(P)
\end{equation}
for every $h \in G^\sigma$ and every $P \in \mathfrak{p}$.  Furthermore,
\begin{equation} \label{general_tauF}
s_{\bar{g} \cdot \eta} (F_{\bar{g}}(P)) = F_{\bar{g}}(-P)
\end{equation}
for every $P \in \mathfrak{p}$.  These identities are summarized in the following pair of diagrams, the first of which commutes for every $h \in G^\sigma$, and the second of which commutes for every $\bar{g} \in G$.
\begin{center}
\begin{tikzpicture}
\node (p) {$G \times \mathfrak{p}$};
\node[below=1.5cm of p] (p2) {$G \times \mathfrak{p}$}; 
\node[right=1.5cm of p] (S) {$\mathcal{S}$};
\node[right=1.5cm of p2] (S2) {$\mathcal{S}$};
\draw[-latex]
(p) edge node (ptoS) [above]{$f$} (S)
(p2) edge node (p2toS2) [below]{$f$} (S2)
(p) edge node (ptop2) [left]{$\Psi_h$} (p2)
(S) edge node (StoS2) [right]{$\Phi_h$} (S2);
\end{tikzpicture}
\hspace{1.5cm}
\begin{tikzpicture}
\node (p) {$\mathfrak{p}$};
\node[below=1.5cm of p] (p2) {$\mathfrak{p}$}; 
\node[right=1.5cm of p] (S) {$\mathcal{S}$};
\node[right=1.5cm of p2] (S2) {$\mathcal{S}$};
\draw[-latex]
(p) edge node (ptoS) [above]{$F_{\bar{g}}$} (S)
(p2) edge node (p2toS2) [below]{$F_{\bar{g}}$} (S2)
(p) edge node (ptop2) [left]{$d\sigma$} (p2)
(S) edge node (StoS2) [right]{$s_{\bar{g}\cdot\eta}$} (S2);
\end{tikzpicture}
\end{center}
Here, we have denoted $f(\bar{g},P)=F_{\bar{g}}(P)$, $\Psi_h(\bar{g},P) = (h\bar{g}h^{-1}, hPh^{-1})$, and $\Phi_h(u) = h \cdot u$.

More generally, one may consider replacing the exponential map in~(\ref{generalexp}) with any retraction $R : \mathfrak{g} \rightarrow G$~\cite{absil2009optimization}.  For instance, if $G$ is a quadratic matrix group, one may choose $R$ equal to the Cayley transform, or more generally, any diagonal Pad\'e approximant of the matrix exponential~\cite{celledoni2000exponential}.

\section{Interpolation on Symmetric Spaces} \label{sec:interp}

In this section, we exploit the correspondence between symmetric spaces and Lie triple systems discussed in Sections~\ref{sec:correspondence}-\ref{sec:correspondence_generalizations} in order to interpolate functions which take values in a symmetric space.

\subsection{A Structure-Preserving Interpolant}

Consider the task of interpolating $m$ elements $u_1,u_2,\dots,u_m \in \mathcal{S}$, which we will think of as the values of a smooth function $u : \Omega \rightarrow \mathcal{S}$ defined on a domain $\Omega \subset \mathbb{R}^d$, $d \ge 1$, at locations $x^{(1)},x^{(2)},\dots,x^{(m)} \in \Omega$.  Our goal is thus to construct a function $\mathcal{I}u : \Omega \rightarrow \mathcal{S}$ that satisfies $\mathcal{I}u(x^{(i)}) = u_i$, $i=1,2,\dots,m$, and has a desired level of regularity (e.g., continuity).  We assume in what follows that for each $x \in \Omega$, $u(x)$ belongs to the range of the map~(\ref{bijection}).  We may then interpolate $u_1,u_2,\dots,u_m$ by interpolating $F^{-1}(u_1), F^{-1}(u_2),\dots,F^{-1}(u_m) \in \mathfrak{p}$ and mapping the result back to $\mathcal{S}$ via $F$.  More precisely, set
\[
\mathcal{I}u(x) = F(\hat{\mathcal{I}}P(x)),
\]
where $P(x) = F^{-1}(u(x))$ and $\hat{\mathcal{I}}P : \Omega \rightarrow \mathfrak{p}$ is an interpolant of $F^{-1}(u_1), F^{-1}(u_2),\dots,F^{-1}(u_m)$.  Then $\mathcal{I}u$ interpolates the data while fulfilling the following important properties.

\begin{proposition} \label{prop:interp_equivariance}
Suppose that $\hat{\mathcal{I}}$ commutes with $\mathrm{Ad}_g$ for every $g \in G^\sigma$.  That is,
\[
\hat{\mathcal{I}}(g P g^{-1}) (x) = g \hat{\mathcal{I}} P(x) g^{-1}
\]
for every $x \in \Omega$ and every $g \in G^\sigma$.
Then $\mathcal{I}$ is $G^\sigma$-equivariant.  That is,
\[
\mathcal{I} (g \cdot u)(x) = g \cdot \mathcal{I} u(x)
\]
for every $x \in \Omega$ and every $g \in G^\sigma$ sufficiently close to the identity.
\end{proposition}
\begin{proof}
The claim is a straightforward consequence of Lemma~\ref{lemma:equivariance}.
\end{proof}

\begin{proposition} \label{prop:involution_equivariance}
Suppose that $\hat{\mathcal{I}}$ commutes with $d\sigma\big|_{\mathfrak{p}}$.  That is, 
\[
\hat{\mathcal{I}}(-P)(x) = -\hat{\mathcal{I}}P(x)
\]
for every $x \in \Omega$.  Then $\mathcal{I}$ commutes with $s_\eta$.  That is,
\[
\mathcal{I} (s_\eta(u))(x) = s_\eta(\mathcal{I}u(x))
\]
for every $x \in \Omega$.
\end{proposition}
\begin{proof}
The claim is a straightforward consequence of~(\ref{tauF}).
\end{proof}

The preceding propositions apply, in particular, to any interpolant $\hat{\mathcal{I}}P : \Omega \rightarrow \mathfrak{p}$ of the form
\[
\hat{\mathcal{I}}P(x) = \sum_{i=1}^m \phi_i(x) P(x^{(i)})
\]
with scalar-valued shape functions $\phi_i : \Omega \rightarrow \mathbb{R}$, $i=1,2,\dots,m$, satisfying $\phi_i(x^{(j)}) = \delta_{ij}$, where $\delta_{ij}$ denotes the Kronecker delta.  By the propositions above, such an interpolant gives rise to a $G^\sigma$-equivariant interpolant $\mathcal{I}u : \Omega \rightarrow \mathcal{S}$ that commutes with $s_\eta$, given by
\begin{equation} \label{interp}
\mathcal{I}u(x) = F\left( \sum_{i=1}^m \phi_i(x) F^{-1}(u_i) \right).
\end{equation}
Written more explicitly,
\begin{equation} \label{interp_explicit1}
\mathcal{I}u(x) = \exp(P(x)) \cdot \eta,
\end{equation}
where
\begin{equation} \label{interp_explicit2}
P(x) = \sum_{i=1}^m \phi_i(x) F^{-1}(u_i).
\end{equation}

\subsection{Derivatives of the Interpolant}

The relations~(\ref{interp_explicit1}-\ref{interp_explicit2}) lead to an explicit formula for the derivatives of $\mathcal{I}u(x)$ with respect to each of the coordinate directions $x_j$, $j=1,2,\dots,d$.  Namely,
\begin{equation} \label{dIu}
\frac{\partial \mathcal{I}u}{\partial x_j}(x) = \mathrm{dexp}_{P(x)} \frac{\partial P}{\partial x_j}(x) \cdot \eta,
\end{equation}
where
\[
\frac{\partial P}{\partial x_j}(x) = \sum_{i=1}^m \frac{\partial\phi_i}{\partial x_j}(x) F^{-1}(u_i)
\]
and $\mathrm{dexp}_X Y$ denotes the differential of $\exp$ at $X \in \mathfrak{g}$ in the direction $Y \in \mathfrak{g}$.  

An explicit formula for $\mathrm{dexp}_X Y$ is the series
\[
\mathrm{dexp}_X Y = \exp(X) \sum_{k=0}^\infty \frac{(-1)^k}{(k+1)!} \mathrm{ad}_X^k Y,
\]
where $\mathrm{ad}_X Y = [X,Y]$ denotes the adjoint action of $\mathfrak{g}$ on itself.  In practice, one may truncate this series to numerically approximate $\mathrm{dexp}_X Y$. Note that while the exact value of $\mathrm{dexp}_X Y$ belongs to $\mathfrak{p}$ whenever $X,Y \in \mathfrak{p}$, this need not be true of its truncated approximation. However, this is of little import since any spurious $\mathfrak{k}$-components in such a truncation act trivially on $\eta$ in~(\ref{dIu}).

While the series expansion of $\mathrm{dexp}_X Y$ is valid on any finite-dimensional Lie group, more efficient methods are available for the computation of $\mathrm{dexp}_X Y$ when $G$ is a matrix group.  Arguably the simplest is to make use of the identity~\cite{mathias1996chain,higham2008functions}
\begin{equation} \label{blockexp}
\exp \begin{pmatrix} X & Y \\ 0 & X \end{pmatrix} = \begin{pmatrix} \exp(X) & \mathrm{dexp}_X Y \\ 0 & \exp(X) \end{pmatrix}.
\end{equation}
More sophisticated approaches with better numerical properties can be found in~\cite{al2009computing,higham2008functions}.

The identity~(\ref{blockexp}) can be leveraged to derive formulas for higher-order derivatives of $\mathcal{I}u(x)$, provided of course that $G$ is a matrix group.  As shown in Appendix~\ref{app:d2exp}, we have
\begin{equation} \label{d2Iu}
\frac{\partial^2 \mathcal{I} u}{\partial x_j \partial x_k}(x) = A \cdot \eta
\end{equation}
for each $j,k=1,2,\dots,d$, where $A$ denotes the $(1,4)$ block of the matrix
\[
\mathrm{exp}\begin{pmatrix} X & Y & Z & W \\ 0 & X & 0 & Z \\ 0 & 0 & X & Y \\ 0 & 0 & 0 & X \end{pmatrix},
\]
and $X = P(x)$, $Y = \frac{\partial P}{\partial x_j}(x)$, $Z = \frac{\partial P}{\partial x_k}(x)$, and $W = \frac{\partial^2 P}{\partial x_j \partial x_k}(x)$.

\subsection{Generalizations} \label{sec:interp_generalizations}

More generally, by fixing an element $\bar{g} \in G$ and adopting the map~(\ref{generalbijection}) instead of $F$, we obtain interpolation schemes of the form
\begin{equation} \label{generalinterp}
\mathcal{I}_{\bar{g}} u(x) = F_{\bar{g}}\left( \sum_{i=1}^m \phi_i(x) F_{\bar{g}}^{-1}(u_i) \right) = \bar{g} \exp\left( \sum_{i=1}^m \phi_i(x) F_{\bar{g}}^{-1}(u_i) \right) \cdot \eta.
\end{equation}
Here, we must of course assume that $u_i$ belongs to the range of $F_{\bar{g}}$ for each $i=1,2,\dots,m$.  This interpolant is therefore suitable for interpolating elements of $\mathcal{S}$ in a neighborhood of $\bar{g} \cdot \eta$.
Using the fact that $F_{hg}(P) = h \cdot F_{g}(P)$ for every $h, g \in G$ and every $P \in \mathfrak{p}$, one finds that this interpolant is equivariant under the action of the full group $G$, in the sense that
\begin{equation} \label{general_equivariance}
\mathcal{I}_{h\bar{g}} (h \cdot u)(x) = h \cdot \mathcal{I}_{\bar{g}} u(x)
\end{equation}
for every $x \in \Omega$ and every $h \in G$ sufficiently close to the identity.  On the other hand, the equivariance of $F_{\bar{g}}$ under the action of the subgroup $G^\sigma$ (recall~(\ref{generalequivariance})) implies that
\begin{equation} \label{general_equivariance_subgroup}
\mathcal{I}_{h\bar{g}h^{-1}} (h \cdot u)(x) = h \cdot \mathcal{I}_{\bar{g}} u(x)
\end{equation}
for every $x \in \Omega$ and every $h \in G^\sigma$ sufficiently close to the identity.  Comparing~(\ref{general_equivariance}) with~(\ref{general_equivariance_subgroup}) leads to the conclusion that this interpolant is invariant under post-multiplication of $\bar{g}$ by elements of $G^\sigma$; that is,
\begin{equation} \label{postmult}
\mathcal{I}_{\bar{g}h} u(x) = \mathcal{I}_{\bar{g}} u(x)
\end{equation}
for every $x \in \Omega$ and every $h \in G^\sigma$ sufficiently close to the identity. 
Finally, as a consequence of~(\ref{general_tauF}),
\[
\mathcal{I}_{\bar{g}} (s_{\bar{g} \cdot \eta} (u))(x) = s_{\bar{g} \cdot \eta} (\mathcal{I}_{\bar{g}} u(x))
\]
for every $x \in \Omega$.

A natural choice for $\bar{g}$ is not immediately evident, but one heuristic is to select $j \in \{1,2,\dots,m\}$ and set $\bar{g}$ equal to a representative of the coset  $ \bar{\varphi}^{-1} (u_j)$.  
A more interesting option is to allow $\bar{g}$ to vary with $x$ and to define $\bar{g}(x)$ implicitly via
\begin{equation} \label{Riemannianmean}
\bar{g}(x) \cdot \eta = \mathcal{I}_{\bar{g}(x)} u(x).
\end{equation}
Equivalently,
\begin{equation} \label{Riemannianmean2}
\sum_{i=1}^m \phi_i(x) F_{\bar{g}(x)}^{-1}(u_i) = 0.
\end{equation}
In analogy with~(\ref{general_equivariance}), the interpolant $\mathcal{I}_{\bar{g}(x)} u(x)$ so defined is equivariant with respect to the action of the full group $G$, not merely the subgroup $G^\sigma$.  That is,
\begin{equation} \label{Gequivariant}
\mathcal{I}_{h\bar{g}(x)} (h \cdot u)(x) = h \cdot \mathcal{I}_{\bar{g}(x)} u(x)
\end{equation}
for every $x \in \Omega$ and every $h \in G$ sufficiently close to the identity.

A method for computing the interpolant $\mathcal{I}_{\bar{g}(x)} u(x)$ numerically is self-evident.  Namely, one performs the fixed-point iteration suggested by~(\ref{Riemannianmean}),  
as we explain in greater detail in Section~\ref{sec:applications}.

We show below that if $G$ is equipped with a left-invariant Riemannian metric for which the restrictions to $\mathfrak{p}$ of the Lie group exponential and Riemannian exponential maps coincide, then~(\ref{Riemannianmean2}) characterizes the coset $[\bar{g}(x)] \in G/G^\sigma$ as the weighted Riemannian mean of the cosets $[g_1],[g_2],\dots,[g_m] \in G/G^\sigma$, where $g_i \cdot \eta = u_i$, $i=1,2,\dots,m$.  The statement of this lemma makes use of the following observation.  Any left-invariant Riemannian metric on $G$ is uniquely defined by an inner product on $\mathfrak{g}$.  The restriction of this inner product to $\mathfrak{p}$ induces a left-invariant metric on $G/G^\sigma$ by virtue of the isomorphism $\mathfrak{p} \cong \mathfrak{g}/\mathfrak{k} = T_{[e]} (G/G^\sigma)$.

\begin{lemma} \label{lemma:Riemannian_mean}
Let $G$ be equipped with a left-invariant Riemmannian metric.  For each $g \in G$, denote by $\mathrm{Exp}_g : T_g G \rightarrow G$ the corresponding Riemannian exponential map.   Suppose that 
\begin{equation} \label{Expexp}
\left.\mathrm{Exp}_e\right|_{\mathfrak{p}} = \left.\exp\right|_{\mathfrak{p}}.
\end{equation}
If $\bar{g}(x) \in G$ is a solution of~(\ref{Riemannianmean}) (or, equivalently, (\ref{Riemannianmean2})), then $[\bar{g}(x)]\in G/G^\sigma$ locally minimizes
\begin{equation} \label{mindist}
\sum_{i=1}^m \phi_i(x) \, \mathrm{dist}\left([h],[g_i]\right)^2
\end{equation}
among all $[h] \in G/G^\sigma$, where $\mathrm{dist} : G/G^\sigma \times G/G^\sigma \rightarrow \mathbb{R}$ denotes the induced geodesic distance on $G/G^\sigma$, and $g_1,g_2,\dots,g_m \in G$ satisfy $g_i \cdot \eta = u_i$, $i=1,2,\dots,m$.
\end{lemma}
\begin{proof}
For each $i=1,2,\dots,m$, let $P_i = F_{\bar{g}(x)}^{-1}(u_i)$, so that $\bar{g}(x) \exp(P_i) \cdot \eta = u_i$.  Then, as cosets in $G/G^\sigma$, we have $[\bar{g}(x) \exp(P_i)] = [g_i]$.  Define $c(t) = \bar{g}(x) \exp(tP_i)$ for $0 \le t \le 1$.  Since $\bar{g}\exp(tP_i)$ coincides with $\bar{g} \, \mathrm{Exp}_e(tP_i) = \mathrm{Exp}_{\bar{g}}(tP_i)$, the curve $c(t)$ is a geodesic on $G$.  Moreover, since $P_i \in \mathfrak{p}$, the tangent vector $c'(t)$ to this curve is everywhere horizontal.  Thus, $[c(t)]$ is a geodesic in $G/G^\sigma$ satisfying $[c(0)] = [\bar{g}(x)]$ and $[c(1)] = [\bar{g}(x) \exp(P_i)] = [g_i]$.  This shows that~(\ref{Riemannianmean2}) is equivalent to
\[
\sum_{i=1}^m \phi_i(x) \mathrm{Exp}_{[\bar{g}(x)]}^{-1} [g_i] = 0,
\]
where $\mathrm{Exp}_{[\bar{g}(x)]} : T_{[\bar{g}(x)]} (G/G^\sigma) \rightarrow G/G^\sigma$ denotes the Riemannian exponential map at $[\bar{g}(x)] \in G/G^\sigma$.  The latter equation is precisely the equation which characterizes minimizers of~(\ref{mindist}); see~\cite[Theorem 1.2]{karcher1977riemannian}.
\end{proof}

Lemma~\ref{lemma:Riemannian_mean} applies, in particular, if $G$ is equipped with a bi-invariant metric, since in that setting the Lie group exponential and Riemannian exponential maps coincide.  Notice that minimizers of~(\ref{mindist}) are precisely geodesic finite elements on $G/G^\sigma$, as described in~\cite{grohs2015optimal,grohs2013quasi,sander2012geodesic,sander2015geodesic}.  We refer the reader to those articles for further information about the approximation properties of these interpolants, as well as the convergence properties of iterative algorithms used to compute them.

\section{Applications} \label{sec:applications}

In this section, we apply the general theory above to several symmetric spaces, including the space of symmetric positive-definite matrices, the space of Lorentzian metrics, and the Grassmannian.  

\subsection{Symmetric Matrices with Fixed Signature} \label{sec:opq}

Let $n$ be a positive integer and let $p$ and $q$ be nonnegative integers satisfying $p+q=n$.  Consider the set 
\[
\mathcal{L} = \{L \in \mathbb{R}^{n \times n} \mid L=L^T, \, \det L \neq 0, \, \mathrm{signature}(L)=(q,p) \},
\]
where $\mathrm{signature}(L)$ denotes the signature of a nonsingular symmetric matrix $L$ -- an ordered pair indicating the number of positive and negative eigenvalues of $L$.
The general linear group $GL_n(\mathbb{R})$ acts transitively on $\mathcal{L}$ via the group action
\[
A \cdot L = ALA^T,
\]
where $A \in GL_n(\mathbb{R})$ and $L \in \mathcal{L}$.  Let $J = \mathrm{diag}(-1,\dots,-1,1,\dots,1)$ denote the diagonal $n \times n$ matrix with $p$ entries equal to $-1$ and $q$ entries equal to $1$.  The stabilizer of $J$ in $GL_n(\mathbb{R})$ is the indefinite orthogonal group
\[
O(p,q) = \{Q \in GL_n(\mathbb{R}) \mid Q J Q^T = J\}.
\]
Its elements are precisely those matrices that are fixed points of the involutive automorphism
\begin{align*}
\sigma : GL_n(\mathbb{R}) &\rightarrow GL_n(\mathbb{R}) \\
A &\mapsto J A^{-T} J,
\end{align*}
where $A^{-T}$ denotes the inverse transpose of a matrix $A \in GL_n(\mathbb{R})$.  In contrast, the set of matrices which are mapped by $\sigma$ to their inverses is
\[
Sym_J(n) = \{P \in GL_n(\mathbb{R}) \mid PJ = JP^T \}.
\]

The setting we have just described is an instance of the general theory presented in Section~\ref{sec:notation}, with $G = GL_n(\mathbb{R})$, $G^\sigma = O(p,q)$, $G_\sigma = Sym_J(n)$, $\mathcal{S}=\mathcal{L}$, and $\eta=J$.
It follows that the generalized polar decomposition~(\ref{gpd}) of a matrix $A \in GL_n(\mathbb{R})$ (sufficiently close to the identity matrix $I$) with respect to $\sigma$ reads
\begin{equation} \label{gpd_opq}
A = PQ, \quad P \in Sym_J(n), \, Q \in O(p,q).
\end{equation}
The Cartan decomposition~(\ref{cartan}) decomposes an element $Z$ of the Lie algebra $\mathfrak{gl}_n(\mathbb{R}) = \mathbb{R}^{n \times n}$ of the general linear group as a sum
\[
Z = X+Y, \quad X \in \mathfrak{sym}_J(n), \, Y \in \mathfrak{o}(p,q),
\]
where
\[
\mathfrak{sym}_J(n) = \{ X \in \mathfrak{gl}_n(\mathbb{R}) \mid XJ = JX^T\}
\]
and
\[
\mathfrak{o}(p,q) = \{Y \in \mathfrak{gl}_n(\mathbb{R}) \mid YJ + JY^T = 0 \}
\]
denotes the Lie algebra of $O(p,q)$.

We can now write down the map $F : \mathfrak{sym}_J(n) \rightarrow \mathcal{L}$ defined abstractly in~(\ref{bijection}), which provides a diffeomorphism between a neighborhood of the zero matrix and a neighborhood of $J$.  By definition,
\begin{align} 
F(X) 
&= \exp(X) J \exp(X)^T \nonumber \\
&=  \exp(X) \exp(X) J \nonumber \\
&= \exp(2X) J, \label{F_opq}
\end{align}
where the second line follows from the fact that $\exp(X) \in Sym_J(n)$ whenever $X \in \mathfrak{sym}_J(n)$.

The inverse of $F$ can likewise be expressed in closed form.  This can be obtained directly by solving~(\ref{F_opq}) for $X$, but it is instructive to see how to derive the same result by inverting each of the maps appearing in the composition~(\ref{bijection}).  To start, note that explicit formulas for the matrices $P$ and $Q$ in the decomposition~(\ref{gpd_opq}) of a matrix $A \in GL_n(\mathbb{R})$ are known~\cite{higham2010canonical}.  Provided that $AJA^T J$ has no negative real eigenvalues, we have
\begin{align*}
P &= (AJA^T J)^{1/2}, \\
Q &= (AJA^T J)^{-1/2} A,
\end{align*}
where $B^{1/2}$ denotes the principal square root of a matrix $B$, and $B^{-1/2}$ denotes the inverse of $B^{1/2}$.  Thus, if $A \cdot J = AJA^T =  L \in \mathcal{L}$ and if $LJ$ has no negative real eigenvalues, then the factor $P$ in the polar decomposition~(\ref{gpd_opq}) of $A$ is given by
\[
P = (LJ)^{1/2}.
\]
It follows that for such a matrix $L$,
\begin{equation*} 
F^{-1}(L) = \log\left( (LJ)^{1/2} \right),
\end{equation*}
where $\log(B)$ denotes the principal logarithm of a matrix $B$.  We henceforth denote by $\mathcal{L}_*$ the set of matrices $L \in \mathcal{L}$ for which $LJ$ has no negative real eigenvalues, so that $F^{-1}(L)$ is well-defined for $L \in \mathcal{L}_*$.

The right-hand side of~(\ref{interp1_opq}) can be simplified using the following property of the matrix logarithm, whose proof can be found in~\cite[Theorem 11.2]{higham2008functions}: If a square matrix $B$ has no negative real eigenvalues, then
\[
\log(B^{1/2}) = \frac{1}{2}\log(B).
\]
From this it follows that
\begin{equation} \label{Finv_opq}
F^{-1}(L) = \frac{1}{2}\log\left( LJ \right)
\end{equation}
for $L \in \mathcal{L}_*$.  This formula, of course, could have been obtained directly from~(\ref{F_opq}), but we have chosen a more circuitous derivation to give a concrete illustration of the theory presented in Section~\ref{sec:symspace}. 

Substituting~(\ref{F_opq}) and~(\ref{Finv_opq}) into~(\ref{interp}) gives the following heuristic for interpolating a set of matrices $L_1,L_2,\dots,L_m \in \mathcal{L}_*$ -- thought of as the values of a smooth function $L : \Omega \rightarrow \mathcal{L}_*$ at points $x^{(1)},x^{(2)},\dots,x^{(m)}$ in a domain $\Omega$ -- at a point $x \in \Omega$:
\begin{equation} \label{interp1_opq}
\mathcal{I}L(x) = \exp\left( \sum_{i=1}^m \phi_i(x) \log\left( L_i J \right) \right) J.
\end{equation}
Here, as before, the functions $\phi_i : \Omega \rightarrow \mathbb{R}$, $i=1,2,\dots,m$, denote scalar-valued shape functions with the property that $\phi_i(x^{(j)}) = \delta_{ij}$.  Using the identity $J^{-1}=J$ together with the fact that the matrix exponential and logarithm commute with conjugation, the right-hand side of~(\ref{interp1_opq}) can be written equivalently as
\begin{equation} \label{interp2_opq}
\mathcal{I}L(x) = J \exp\left( \sum_{i=1}^m \phi_i(x) \log\left( J L_i \right) \right).
\end{equation}

In addition to satisfying $\mathcal{I}L(x) \in \mathcal{L}$ for every $x \in \Omega$, the interpolant so defined enjoys the following properties, which generalize the observations made in Theorems 3.13 and 4.2 of~\cite{arsigny2007geometric}.

\begin{lemma} \label{lemma:equivariance_opq}
Let $Q \in O(p,q)$.  If $\tilde{L}_i = Q L_i Q^T$, $i=1,2,\dots,m$, and if $Q$ is sufficiently close to the identity matrix, then
\[
\mathcal{I}\tilde{L}(x) = Q \, \mathcal{I}L(x) \, Q^T.
\]
for every $x \in \Omega$.
\end{lemma}
\begin{proof}
Apply Proposition~\ref{prop:interp_equivariance}.
\end{proof}

\begin{lemma}
If $\tilde{L}_i = JL_i^{-1}J$, $i=1,2,\dots,m$, then
\[
\mathcal{I}\tilde{L}(x) = J\left(\mathcal{I}L(x)\right)^{-1}J.
\]
for every $x \in \Omega$.
\end{lemma}
\begin{proof}
Apply Proposition~\ref{prop:involution_equivariance}, noting that $s_\eta(L) = JL^{-1}J$ for $L \in \mathcal{L}$.
\end{proof}
Note that the preceding two propositions can be combined to conclude that if $\tilde{L}_i = L_i^{-1}$, $i=1,2,\dots,m$, then
\[
\mathcal{I}\tilde{L}(x) = \left(\mathcal{I}L(x)\right)^{-1}.
\]
To see this, observe that $L_i^{-1} = J (JL_i^{-1}J) J^T$ and $J \in O(p,q)$.

\begin{lemma}
If $\sum_{i=1}^m \phi_i(x) = 1$ for every $x \in \Omega$, then
\[
\det \mathcal{I}L(x) = \prod_{i=1}^m \left( \det L_i \right)^{\phi_i(x)}
\]
for every $x \in \Omega$.
\end{lemma}
\begin{proof}
Using the identities $\det\exp(A) = \exp(\mathrm{tr}(A))$ and $\mathrm{tr}(\log(A)) = \log(\det A)$, we have
\begin{align*}
\det \mathcal{I}L(x) 
&= \det \left( \exp\left( \sum_{i=1}^m \phi_i(x) \log( L_i J ) \right) \right) \det J \\
&= \exp \left( \mathrm{tr} \left( \sum_{i=1}^m \phi_i(x) \log( L_i J ) \right) \right) \det J \\
&= \exp \left( \sum_{i=1}^m \phi_i(x) \mathrm{tr}  \left( \log( L_i J ) \right) \right) \det J \\
&= \exp \left( \sum_{i=1}^m \phi_i(x) \log \left( \det( L_i J ) \right) \right) \det J \\
&= \left( \prod_{i=1}^m \det( L_i J )^{\phi_i(x)} \right) \det J \\
&= \left( \prod_{i=1}^m \det( L_i )^{\phi_i(x)} \det( J )^{\phi_i(x)} \right) \det J \\
\end{align*}
The conclusion then follows from the facts that $\sum_{i=1}^m \phi_i(x) = 1$ and $\det J = \pm 1$.
\end{proof}

\paragraph{Generalizations.}  

As explained abstractly in Section~\ref{sec:interp_generalizations}, the interpolation formula~(\ref{interp2_opq}) can be generalized by fixing an element $\bar{A} \in GL_n(\mathbb{R})$ and replacing~(\ref{F_opq}) with the map
\[
F_{\bar{A}}(X) = \bar{A} \exp(X) J \left(\bar{A} \exp(X)\right)^T  = \bar{A} F(X) \bar{A}^T.
\]
The inverse of this map reads
\[
F_{\bar{A}}^{-1} (L) = \frac{1}{2} \log( \bar{A}^{-1} L \bar{A}^{-T} J ).
\]
Substituting into~(\ref{generalinterp}), gives, after some simplification, the interpolation formula
\begin{equation} \label{interp_general_opq}
\mathcal{I}_{\bar{A}} L(x) = \bar{L} \exp\left( \sum_{i=1}^m \phi_i(x) \log\left( \bar{L}^{-1} L_i \right) \right),
\end{equation}
where $\bar{L} = \bar{A} J \bar{A}^T$.  

Rather than fixing $\bar{A}$, one may choose to define $\bar{A}$ implicitly via~(\ref{Riemannianmean}); that is,
\[
\bar{A}(x) J \bar{A}(x)^T = \mathcal{I}_{\bar{A}(x)} L(x).
\]
The output of the resulting interpolation scheme is the solution $\bar{L}$ to the equation
\begin{equation} \label{Riemannian_mean_opq}
\sum_{i=1}^m \phi_i(x) \log\left( \bar{L}^{-1} L_i \right)  = 0,
\end{equation}
which can be computed with a fixed-point iteration.

\paragraph{Algorithms.}

In summary, we have derived the following pair of algorithms for interpolating matrices in the space $\mathcal{L}$ of nonsingular symmetric matrices with signature $(q,p)$.  The first of these algorithms implements~(\ref{interp_general_opq}), which reduces to~(\ref{interp2_opq}) when $\bar{L}$ is taken equal to $J$.  The algorithm implicitly requires its inputs to have the property that for each $i=1,2,\dots,m$, the matrix $\bar{L}^{-1} L_i$ has no negative real eigenvalues.

\begin{algorithm}[H] 
\caption{Interpolation of symmetric matrices with fixed signature} \label{alg:opq}
\begin{algorithmic}[1]
\Require Matrices $\{L_i \in \mathcal{L}\}_{i=1}^m$, shape functions $\{\phi_i : \Omega \rightarrow \mathbb{R} \}_{i=1}^m$, point $x \in \Omega$, matrix~$\bar{L} \in \mathcal{L}$ 
\State\Return $\bar{L} \exp\left( \sum_{i=1}^m \phi_i(x) \log\left( \bar{L}^{-1} L_i \right) \right)$
\end{algorithmic}
\end{algorithm}

The second algorithm solves~(\ref{Riemannian_mean_opq}), and requires the same constraint on its inputs as Algorithm~\ref{alg:opq}.  Observe that Algorithm~\ref{alg:opq} is equivalent to Algorithm~\ref{alg:opq_Riemannian} if one terminates the fixed-point iteration after the first iteration.

\begin{algorithm}[H] 
\caption{Iterative interpolation of symmetric matrices with fixed signature} \label{alg:opq_Riemannian}
\begin{algorithmic}[1]
\Require Matrices $\{L_i \in \mathcal{L}\}_{i=1}^m$, shape functions $\{\phi_i : \Omega \rightarrow \mathbb{R} \}_{i=1}^m$, point $x \in \Omega$, initial guess~$\bar{L} \in \mathcal{L}$, tolerance $\varepsilon > 0$
\While{$\left\| \sum_{i=1}^m \phi_i(x) \log\left( \bar{L}^{-1} L_i \right) \right\| > \varepsilon$ \vspace{0.2em}}
\State $\bar{L}= \bar{L} \exp\left( \sum_{i=1}^m \phi_i(x) \log\left( \bar{L}^{-1} L_i \right) \right)$
\EndWhile
\State\Return $\bar{L}$
\end{algorithmic}
\end{algorithm}

\subsubsection{Symmetric Positive-Definite Matrices} \label{sec:spd}

When $J=I$, the preceding theory provides structure-preserving interpolation schemes for the space $SPD(n)$ of symmetric positive-definite matrices.  The formula~(\ref{interp2_opq}) is the weighted log-Euclidean mean introduced by~\cite{arsigny2007geometric}, and equation~(\ref{Riemannian_mean_opq}) gives the weighted Riemannian mean (or Karcher mean) of symmetric positive-definite matrices~\cite{moakher2005differential,bhatia2013riemannian,karcher1977riemannian}.  The latter observation can be viewed as a consequence of Lemma~\ref{lemma:Riemannian_mean}, which applies in this setting for the following reason.  If the general linear group is equipped with the canonical left-invariant Riemannian metric, then the Riemmannian exponential map $\mathrm{Exp}_I : T_I GL_n(\mathbb{R}) \rightarrow GL_n(\mathbb{R})$ at the identity reads~\cite{andruchow2014left}
\[
\mathrm{Exp}_I(A) = \exp(A^T) \exp(A-A^T).
\]
This formula reduces to $\mathrm{Exp}_I(A) = \exp(A)$ when $A$ is a normal matrix (i.e., $A^T A = A A^T$).  In particular, $\mathrm{Exp}_I(A) = \exp(A)$ when $A$ is symmetric, which is precisely the condition~(\ref{Expexp}).

We remark that the interpolation formula~(\ref{interp2_opq}) on $SPD(n)$ was devised in~\cite{arsigny2007geometric} by endowing $SPD(n)$ with what the authors term a ``novel vector space structure.''  This vector space structure is nothing more than that obtained by identifying $SPD(n)$ with the Lie triple system $\mathfrak{sym}_I(n)$ via the map~(\ref{Finv_opq}), as we have done here.

\subsubsection{Lorentzian Metrics} \label{sec:lorentz}

When $n=4$ and $J = \mathrm{diag}(-1,1,1,1)$, the preceding theory provides structure-preserving interpolation schemes for the space of Lorentzian metrics -- the space of symmetric, nonsingular matrices having signature $(3,1)$.  Lemma~\ref{lemma:equivariance_opq} states that the interpolation operator~(\ref{interp2_opq}) in this setting commutes with Lorentz transformations.  By choosing, for instance, $\Omega$ equal to a four-dimensional simplex (or a four-dimensional hypercube) and $\{\phi_i\}_i$ equal to scalar-valued Lagrange polynomials (or tensor products of Lagrange polynomials) on $\Omega$, one obtains a family of Lorentzian metric-valued finite elements. 

In view of their potential application to numerical relativity, we have numerically computed the interpolation error committed by such elements when approximating the Schwarzschild metric, which is an explicit solution to Einstein's equations outside of a spherical mass~\cite{carroll2004spacetime}.  In Cartesian coordinates, this metric reads
\begin{equation} \label{schwarz}
L(t,x,y,z) = 
\begin{pmatrix}
-\left(1 - \frac{R}{r}\right) & 0 & 0 & 0 \\
0 & 1+\left(\frac{R}{r-R}\right) \frac{x^2}{r^2} & \left(\frac{R}{r-R}\right) \frac{xy}{r^2} & \left(\frac{R}{r-R}\right) \frac{xz}{r^2} \\
0 & \left(\frac{R}{r-R}\right) \frac{xy}{r^2} & 1+\left(\frac{R}{r-R}\right) \frac{y^2}{r^2} & \left(\frac{R}{r-R}\right) \frac{yz}{r^2} \\
0 & \left(\frac{R}{r-R}\right) \frac{xz}{r^2} & \left(\frac{R}{r-R}\right) \frac{yz}{r^2} & 1+\left(\frac{R}{r-R}\right) \frac{z^2}{r^2}
\end{pmatrix},
\end{equation}
where $R$ (the Schwarzschild radius) is a positive constant (which we take equal to 1 in what follows) and $r=\sqrt{x^2+y^2+z^2} > R$.  We interpolated this metric over the region $U = \{0\} \times [2,3] \times [2,3] \times [2,3]$ on a uniform $N \times N \times N$ grid of cubes using the formula~(\ref{interp2_opq}) elementwise, with shape functions $\{\phi_i\}_i$ given by tensor products of Lagrange polynomials of degree $k$.  The results in Table~\ref{tab:schwarz} indicate that the $L^2$-error
\begin{equation} \label{L2err}
\|\mathcal{I}L - L\|_{L^2(U)} = \left( \int_U \left\| \mathcal{I}L(t,x,y,z) - L(t,x,y,z) \right\|_F^2 \, dx \, dy \, dz \right)^{1/2}
\end{equation}
(which we approximated with numerical quadrature) converges to zero with order 2 and 3, respectively, when using polynomials of degree $k=1$ and $k=2$.  Here, $\|\cdot\|_F$ denotes the Frobenius norm.  In addition, Table~\ref{tab:schwarz} indicates that the $H^1$-error 
\begin{equation} \label{H1err}
|\mathcal{I}L - L|_{H^1(U)} = \left( \int_U \sum_{j=1}^4 \left\| \frac{\partial \mathcal{I}L}{\partial \xi_j}(t,x,y,z) - \frac{\partial L}{\partial \xi_j}(t,x,y,z) \right\|_F^2 \, dx \, dy \, dz \right)^{1/2}
\end{equation}
converges to zero with order 1 and 2, respectively, when using polynomials of degree $k=1$ and $k=2$.  Here, we have denoted $\xi=(t,x,y,z)$.

For the sake of comparison, Table~\ref{tab:schwarzcomponentwise} shows the interpolation errors committed when applying componentwise polynomial interpolation to the same problem.  Within each element, the value of this interpolant at a point $\xi=(t,x,y,z)$ lying in the element is given by
\begin{equation} \label{componentwiseinterp}
\mathcal{I}L(\xi) = \sum_{i=1}^m \phi_i(\xi) L_i,
\end{equation}
where $\{\phi_i\}_i$ are tensor products of Lagrange polynomials of degree $k$ and $\{L_i\}_i$ are the values of $L$ at the corresponding degrees of freedom.
The errors committed by this interpolation scheme are very close to those observed in Table~\ref{tab:schwarz} for the structure-preserving scheme~(\ref{interp2_opq}).

For this particular numerical example, the componentwise polynomial interpolant~(\ref{componentwiseinterp}) has correct signature $(3,1)$ for every $(t,x,y,z) \in U$.  This need not hold in general.  For example, consider the metric tensor
\[
L(t,x,y,z) = 
\begin{pmatrix} 
-6\sin^2(2\pi x)+3\sin^2(\pi x) & 3\cos(2\pi x) & 0 & 0 \\
3\cos(2\pi x) & 2\sin^2(2\pi x)+2\sin^2(\pi x) & 0 & 0 \\
0 & 0 & 1 & 0 \\
0 & 0 & 0 & 1
\end{pmatrix}.
\]
Though not a solution to Einstein's equations, this metric tensor nonetheless has signature $(3,1)$ everywhere. Indeed, a numerical calculation verifies that at all points $(t,x,y,z)$, the matrix $L(t,x,y,z)$ has eigenvalues $\lambda_-,1,1,\lambda_+$ satisfying $\lambda_- \le \alpha$ and $\lambda_+ \ge \beta$ with $\alpha \approx -0.54138$ and $\beta \approx 2.23064$.  Interpolating this metric componentwise with linear polynomials (over the region same region $U$ as above) produces a metric with signature $(4,0)$ at 32 quadrature points (out of 64 total) on the coarsest grid ($N=2$).  The essence of the problem is that for any integer $k$, any $t$, any $y$, and any $z$, the average of $L(t,k/2,y,z)$ and $L(t,(k+1)/2,y,z)$ is
\[
\frac{1}{2} \left( L(t,k/2,y,z) + L(t,(k+1)/2,y,z) \right) = 
\begin{pmatrix}
\frac{3}{2} & 0 & 0 & 0 \\
0 & 1 & 0 & 0 \\
0 & 0 & 1 & 0 \\
0 & 0 & 0 & 1
\end{pmatrix},
\]
which shows (by continuity of the interpolant) that the componentwise linear interpolant~(\ref{componentwiseinterp}) on the coarsest grid ($N=2$) is positive definite on an open subset of $U$.  In contrast, the structure-preserving scheme~(\ref{interp2_opq}) automatically generates an interpolant with correct signature $(3,1)$ at all points $(t,x,y,z)$.

\begin{table}[t]
\centering
\begin{filecontents*}{Data/schwarz.dat}
 2.00000000e+00 3.29635142e-03 2.83970165e-02 1.73962419e-04 2.46940240e-03
 4.00000000e+00 8.38320411e-04 1.42121520e-02 2.17253722e-05 6.20395238e-04
 8.00000000e+00 2.10476342e-04 7.10773420e-03 2.71480654e-06 1.55288106e-04
 1.60000000e+01 5.26752028e-05 3.55407284e-03 3.39321841e-07 3.88338414e-05
\end{filecontents*}
\pgfplotstabletypeset[
every head row/.style={after row=\midrule,before row={\midrule & \multicolumn{4}{c|}{$k=1$} & \multicolumn{4}{c}{$k=2$} \\ \midrule}},
create on use/rate1/.style={create col/dyadic refinement rate={1}},
create on use/rate2/.style={create col/dyadic refinement rate={2}},
create on use/rate3/.style={create col/dyadic refinement rate={3}},
create on use/rate4/.style={create col/dyadic refinement rate={4}},
columns={0,1,rate1,2,rate2,3,rate3,4,rate4},
columns/0/.style={sci zerofill,column type/.add={}{|},column name={$N$}},
columns/1/.style={sci zerofill,precision=3,column type/.add={}{|},column name={$L^2$-error}},
columns/2/.style={sci zerofill,precision=3,column type/.add={}{|},column name={$H^1$-error}}, 
columns/3/.style={sci zerofill,precision=3,column type/.add={}{|},column name={$L^2$-error}},
columns/4/.style={sci zerofill,precision=3,column type/.add={}{|},column name={$H^1$-error}}, 
columns/rate1/.style={fixed zerofill,precision=3,column type/.add={}{|},column name={Order}},
columns/rate2/.style={fixed zerofill,precision=3,column type/.add={}{|},column name={Order}},
columns/rate3/.style={fixed zerofill,precision=3,column type/.add={}{|},column name={Order}},
columns/rate4/.style={fixed zerofill,precision=3,column name={Order}}
]
{Data/schwarz.dat}
\caption{Error incurred when interpolating the Schwarzschild metric~(\ref{schwarz}) over the region $U = \{0\} \times [2,3] \times [2,3] \times [2,3]$ using the formula~(\ref{interp2_opq}).  The interpolant was computed elementwise on a uniform $N \times N \times N$ grid of cubes, with shape functions $\{\phi_i\}_i$ on each cube given by tensor products of Lagrange polynomials of degree $k$.}
\label{tab:schwarz}
\end{table}

\begin{table}[t]
\centering
\begin{filecontents*}{Data/schwarzcomponentwise.dat}
 2.00000000e+00 2.89942429e-03 2.95213085e-02 1.75670947e-04 2.50811061e-03
 4.00000000e+00 7.37247305e-04 1.47651455e-02 2.20146763e-05 6.29589029e-04
 8.00000000e+00 1.85091349e-04 7.38312587e-03 2.75343193e-06 1.57556235e-04
 1.60000000e+01 4.63216330e-05 3.69163190e-03 3.44227727e-07 3.93989763e-05
\end{filecontents*}
\pgfplotstabletypeset[
every head row/.style={after row=\midrule,before row={\midrule & \multicolumn{4}{c|}{$k=1$} & \multicolumn{4}{c}{$k=2$} \\ \midrule}},
create on use/rate1/.style={create col/dyadic refinement rate={1}},
create on use/rate2/.style={create col/dyadic refinement rate={2}},
create on use/rate3/.style={create col/dyadic refinement rate={3}},
create on use/rate4/.style={create col/dyadic refinement rate={4}},
columns={0,1,rate1,2,rate2,3,rate3,4,rate4},
columns/0/.style={sci zerofill,column type/.add={}{|},column name={$N$}},
columns/1/.style={sci zerofill,precision=3,column type/.add={}{|},column name={$L^2$-error}},
columns/2/.style={sci zerofill,precision=3,column type/.add={}{|},column name={$H^1$-error}}, 
columns/3/.style={sci zerofill,precision=3,column type/.add={}{|},column name={$L^2$-error}},
columns/4/.style={sci zerofill,precision=3,column type/.add={}{|},column name={$H^1$-error}}, 
columns/rate1/.style={fixed zerofill,precision=3,column type/.add={}{|},column name={Order}},
columns/rate2/.style={fixed zerofill,precision=3,column type/.add={}{|},column name={Order}},
columns/rate3/.style={fixed zerofill,precision=3,column type/.add={}{|},column name={Order}},
columns/rate4/.style={fixed zerofill,precision=3,column name={Order}}
]
{Data/schwarzcomponentwise.dat}
\caption{Error incurred when interpolating the Schwarzschild metric~(\ref{schwarz}) over the region $U = \{0\} \times [2,3] \times [2,3] \times [2,3]$ using the componentwise interpolation formula~(\ref{componentwiseinterp}).  The interpolant was computed elementwise on a uniform $N \times N \times N$ grid of cubes, with shape functions $\{\phi_i\}_i$ on each cube given by tensor products of Lagrange polynomials of degree $k$.}
\label{tab:schwarzcomponentwise}
\end{table}

\subsection{The Grassmannian} \label{sec:grass}

Let $p$ and $n$ be positive integers satisfying $p < n$.  Consider the Grassmannian $Gr(p,n)$, which consists of all $p$-dimensional linear subspaces of $\mathbb{R}^n$.  Any element $\mathcal{V} \in Gr(p,n)$ can be written as the span of $p$ vectors $v_1,v_2,\dots,v_p \in \mathbb{R}^n$.  The orthogonal group $O(n)$ acts transitively on $Gr(p,n)$ via the action
\[
A \cdot \mathrm{span}(v_1,v_2,\dots,v_p) \mapsto \mathrm{span}(Av_1,Av_2,\dots,Av_p),
\]
where $A \in O(n)$.  For convenience, we will sometimes write $A\mathcal{V}$ as shorthand for \break$\mathrm{span}(Av_1,Av_2,\dots,Av_p)$.   Let $e_1,e_2,\dots,e_n$ be the canonical basis for $\mathbb{R}^n$.  The stabilizer of $\mathrm{span}(e_1,e_2,\dots,e_p)$ in $O(n)$ is the subgroup
\[
O(p) \times O(n-p) = \left\{ \begin{pmatrix} A_1 & 0 \\ 0 & A_2 \end{pmatrix} \mid A_1 \in O(p), \, A_2 \in O(n-p) \right\}.
\]
The elements of $O(p) \times O(n-p)$ are precisely those matrices in $O(n)$ that are fixed points of the involutive automorphism
\begin{align*}
\sigma : O(n) &\rightarrow O(n) \\
A &\mapsto J A J,
\end{align*}
where 
\[
J = \begin{pmatrix} -I_p & 0 \\ 0 & I_{n-p} \end{pmatrix},
\]
and $I_p$ and $I_{n-p}$ denote the $p \times p$ and $(n-p) \times (n-p)$ identity matrices, respectively.  The matrices in $O(n)$ that are mapped to their inverses by $\sigma$ constitute the space
\[
Sym_J(n) \cap O(n) = \{ P \in O(n) \mid PJ = JP^T \}.
\]
The generalized polar decomposition of a matrix $A \in O(n)$ in this setting thus reads
\begin{equation} \label{gpd_grass}
A = PQ, \quad P \in Sym_J(n) \cap O(n), \, Q \in O(p) \times O(n-p).
\end{equation}
The corresponding Cartan decomposition reads
\[
Z = X + Y, \quad X \in \mathfrak{sym}_J(n) \cap \mathfrak{o}(n), \, Y \in \mathfrak{o}(p) \times \mathfrak{o}(n-p),
\]
where, for each $m$, $\mathfrak{o}(m)$ denotes the space of antisymmetric $m \times m$ matrices,
\[
\mathfrak{o}(p) \times \mathfrak{o}(n-p) = \left\{ \begin{pmatrix} Y_1 & 0 \\ 0 & Y_2 \end{pmatrix} \mid Y_1 \in \mathfrak{o}(p), \, Y_2 \in \mathfrak{o}(n-p) \right\}
\]
and
\begin{align*}
\mathfrak{sym}_J(n) \cap \mathfrak{o}(n)
&= \{ X \in \mathfrak{o}(n) \mid XJ = JX^T \} \\
&= \left\{ \begin{pmatrix} 0 & -B^T \\ B & 0 \end{pmatrix} \mid B \in \mathbb{R}^{(n-p) \times p} \right\}.
\end{align*}

The map $F : \mathfrak{sym}_J(n) \cap \mathfrak{o}(n) \rightarrow Gr(p,n)$ is given by
\begin{equation*}
F(X) = \mathrm{span}(\exp(X)e_1, \exp(X)e_2, \dots, \exp(X)e_p).
\end{equation*}
The inverse of $F$ can be computed (naively) as follows. Given an element $\mathcal{V} \in Gr(p,n)$, let $a_1,a_2,\dots,a_p$ be an orthonormal basis for $\mathcal{V}$.  Extend this basis to an orthonormal basis $a_1,a_2,\dots,a_n$ of $\mathbb{R}^n$.  Then
\begin{equation*}
F^{-1}(\mathcal{V}) = \log(P),
\end{equation*}
where $P \in Sym_J(n) \cap O(n)$ is the first factor in the generalized polar decomposition~(\ref{gpd_grass}) of $A = \left( a_1 \, a_2 \, \cdots \, a_n\right)$.  Note that this map is independent of the chosen bases for $\mathcal{V}$ and its orthogonal complement in $\mathbb{R}^n$.  Indeed, if $\tilde{a}_1,\tilde{a}_2,\dots,\tilde{a}_p$ is any other orthonormal basis for $\mathcal{V}$ and $\tilde{a}_{p+1},\tilde{a}_{p+2},\dots,\tilde{a}_n$ is any other basis for the orthogonal complement of $\mathcal{V}$, then there is a matrix $R \in O(p) \times O(n-p)$ such that $\tilde{A} = AR$, where $\tilde{A} = \left( \tilde{a}_1 \, \tilde{a}_2 \, \cdots \, \tilde{a}_n\right)$.  The generalized polar decomposition of $\tilde{A}$ is thus $\tilde{A} = P\tilde{Q}$, where $\tilde{Q} = QR$.

More generally, we may opt to fix an element $\bar{A} \in O(n)$ and consider interpolants of the form~(\ref{generalinterp}) using the map
\begin{equation} \label{F_grass}
F_{\bar{A}} (X) = \mathrm{span}(\bar{A}\exp(X)e_1, \bar{A}\exp(X)e_2, \dots, \bar{A}\exp(X)e_p),
\end{equation}
The inverse of this map, in analogy with the preceding paragraph, is
\begin{equation} \label{Finv_grass}
F^{-1}_{\bar{A}}(\mathcal{V}) = \log(P),
\end{equation}
where now $P \in Sym_J(n) \cap O(n)$ is the first factor in the generalized polar decomposition~(\ref{gpd_grass}) of $\bar{A}^T A$, where  $A \in O(n)$ is a matrix whose first $p$ and last $n-p$ columns, respectively, form orthonormal bases for $\mathcal{V}$ and its orthogonal complement.

\paragraph{Algorithms.}

We now turn our attention to the computation of the interpolant~(\ref{generalinterp}) in this setting. A naive implementation using the steps detailed above for computing $F_{\bar{A}}$ and its inverse would lead to an algorithm for computing the interpolant having complexity $O(n^3)$.  Remarkably, the computation of~(\ref{generalinterp}) can be performed in $O(np^2)$ operations, as we now show.  The resulting algorithm turns out to be identical to that proposed in~\cite{amsallem2008interpolation}.  The fact that this algorithm scales linearly with $n$ is noteworthy, as it renders this interpolation scheme practical for applications in which $n \gg p$.

The derivation of the algorithm hinges upon the following two lemmas, which, when combined, allow for a computation of the interpolant while operating solely on matrices of size $n \times p$ or smaller.  The first lemma gives a useful formula for $F_{\bar{A}} (X)$.

\begin{lemma} \label{lemma:exp_grass}
Let 
\[
\bar{A} = \begin{pmatrix} \bar{A}_1 & \bar{A}_2 \end{pmatrix} \in O(n)
\]
with $\bar{A}_1 \in \mathbb{R}^{n \times p}$ and $\bar{A}_2 \in \mathbb{R}^{n \times (n-p)}$, and let $X =  \begin{pmatrix} 0 & -B^T \\ B & 0 \end{pmatrix} \in \mathfrak{sym}_J(n) \cap \mathfrak{o}(n)$ with $B \in \mathbb{R}^{(n-p) \times p}$.  Then
\[
\bar{A} \exp(X) \begin{pmatrix} I_p \\ 0 \end{pmatrix} = \bar{A}_1 V \cos(\Theta)V^T  + U \sin(\Theta) V^T,
\]
where $U \in \mathbb{R}^{n \times p}$, $\Theta \in \mathbb{R}^{p \times p}$, and $V \in \mathbb{R}^{p \times p}$ denote the factors in the thin singular value decomposition 
\begin{equation} \label{A2Bsvd}
\bar{A}_2 B = U \Theta V^T.
\end{equation}  
In particular, $F_{\bar{A}}(X)$ is the space spanned by the columns of $\bar{A}_1 V \cos(\Theta)V^T  + U \sin(\Theta) V^T$.  Equivalently, since $V$ is orthogonal, $F_{\bar{A}}(X)$ is the space spanned by the columns of $\bar{A}_1 V \cos(\Theta)  + U \sin(\Theta)$.
\end{lemma}
\begin{proof}
The formula is proved in~\cite[Theorem 2.3]{edelman1998geometry}.
\end{proof}

The next lemma gives a useful formula for $F_{\bar{A}}^{-1} (\mathcal{V})$.  Closely related formulas appear without proof in~\cite{begelfor2006affine,amsallem2008interpolation,chang2012feature} and elsewhere, so we give a proof here for completeness.

\begin{lemma} \label{lemma:log_grass}
Let $\bar{A} = \begin{pmatrix} \bar{A}_1 & \bar{A}_2 \end{pmatrix} \in O(n)$ be as in Lemma~(\ref{lemma:exp_grass}), 
and let $\mathcal{V} \in Gr(p,n)$. Let
\[
A = \begin{pmatrix} A_1 & A_2 \end{pmatrix} \in O(n)
\]
be such that the columns of $A_1$ and $A_2$, respectively, form orthonormal bases for $\mathcal{V}$ and its orthogonal complement.  Assume that $\bar{A}_1^T A_1$ is invertible.
Then
\[
F^{-1}_{\bar{A}}(\mathcal{V}) = \begin{pmatrix} 0 & -B^T \\ B & 0 \end{pmatrix},
\]
where
\begin{equation} \label{Bsoln}
B = \bar{A}_2^T U \arctan(\Sigma) V^T,
\end{equation}
and $U \in \mathbb{R}^{n \times p}$, $\Sigma \in \mathbb{R}^{p \times p}$, and $V \in \mathbb{R}^{p \times p}$ denote the factors in the thin singular value decomposition
\begin{equation} \label{svdA1A2}
(I-\bar{A}_1 \bar{A}_1^T) A_1 (\bar{A}_1^T A_1)^{-1} = U \Sigma V^T.
\end{equation}
\end{lemma}
\begin{proof}
It is enough to check  that if $B$ is given by~(\ref{Bsoln}), then the image of $\begin{pmatrix} 0 & -B^T \\ B & 0 \end{pmatrix}$ under $F_{\bar{A}}$ is $\mathcal{V}$.  In other words, we must check that the columns of 
\begin{equation} \label{AbarexpB}
\bar{A} \exp \begin{pmatrix} 0 & -B^T \\ B & 0 \end{pmatrix} \begin{pmatrix} I_p \\ 0 \end{pmatrix}
\end{equation}
span $\mathcal{V}$.  To this end, observe that by the orthogonality of $\bar{A}$,
\begin{equation} \label{A2A2U}
\bar{A}_2 \bar{A}_2^T U = (I - \bar{A}_1 \bar{A}_1^T) U = U,
\end{equation}
where the last equality follows from~(\ref{svdA1A2}) upon noting that $(I - \bar{A}_1 \bar{A}_1^T)$ is a projection.  Thus, by inspection of~(\ref{Bsoln}), the thin singular value decomposition of $\bar{A}_2 B$ is
\[
\bar{A}_2 B = U \Theta V^T,
\]
where $\Theta = \arctan{\Sigma}$.  Now by Lemma~\ref{lemma:exp_grass},
\begin{equation} \label{expgrass2} 
\bar{A} \exp \begin{pmatrix} 0 & -B^T \\ B & 0 \end{pmatrix} \begin{pmatrix} I_p \\ 0 \end{pmatrix}
= \bar{A}_1 V \cos(\Theta)V^T  + U \sin(\Theta) V^T.
\end{equation}
Using~(\ref{svdA1A2}), this simplifies to
\begin{align}
\bar{A} \exp \begin{pmatrix} 0 & -B^T \\ B & 0 \end{pmatrix} \begin{pmatrix} I_p \\ 0 \end{pmatrix}
&= \bar{A}_1 V \cos(\Theta)V^T + (I-\bar{A}_1 \bar{A}_1^T) A_1 (\bar{A}_1^T A_1)^{-1} V \Sigma^{-1} \sin(\Theta) V^T \nonumber \\
&= \bar{A}_1 V \cos(\Theta)V^T + (I-\bar{A}_1 \bar{A}_1^T) A_1 (\bar{A}_1^T A_1)^{-1} V \cos(\Theta) V^T \nonumber \\
&= A_1  (\bar{A}_1^T A_1)^{-1} V \cos(\Theta) V^T. \nonumber
\end{align}
Observe that since $\Sigma=\tan(\Theta)$ is finite, the diagonal entries of $\cos(\Theta)$ are nonzero. Thus, $(\bar{A}_1^T A_1)^{-1} V \cos(\Theta) V^T$ is invertible, so we conclude that the columns of~(\ref{AbarexpB}) span the same space that is spanned by the columns of $A_1$, namely $\mathcal{V}$.
\end{proof}

The preceding two lemmas lead to the following algorithm, which coincides with that introduced in~\cite{amsallem2008interpolation}, for computing the interpolant
\begin{equation} \label{interp_grass}
\mathcal{I}_{\bar{A}}\mathcal{V}(x) = F_{\bar{A}}\left( \sum_{i=1}^m \phi_i(x) F_{\bar{A}}^{-1}(\mathcal{V}^{(i)}) \right)
\end{equation}
of elements $\mathcal{V}^{(1)}, \mathcal{V}^{(2)}, \dots, \mathcal{V}^{(m)}$ of $Gr(p,n)$.  Note that the computational complexity of this algorithm is $O(np^2)$.  In particular, owing to the identity~(\ref{A2A2U}), the $(n-p) \times n$ matrix $\bar{A}_2$ plays no role in the algorithm, despite its worrisome appearance in~(\ref{A2Bsvd}) and~(\ref{Bsoln}).

\begin{algorithm}[H] 
\caption{Interpolation on the Grassmannian $Gr(p,n)$} \label{alg_grass}
\begin{algorithmic}[1]
\Require Subspaces $\{\mathcal{V}^{(i)} \in Gr(p,n)\}_{i=1}^m$, shape functions $\{\phi_i : \Omega \rightarrow \mathbb{R} \}_{i=1}^m$, point $x \in \Omega$, matrix~$\bar{A}_1 \in \mathbb{R}^{n \times p}$ with orthonormal columns
\State $Z \gets 0_{n \times p}$
\For{$i=1,2,\dots,m$}
\State \label{alg_grass_basis} Let $A_1 \in \mathbb{R}^{n \times p}$ be a matrix whose columns form an orthonormal basis for $\mathcal{V}^{(i)}$.
\State \begin{varwidth}[t]{\linewidth} Compute the thin singular value decomposition
\[
(I-\bar{A}_1 \bar{A}_1^T) A_1 (\bar{A}_1^T A_1)^{-1} = U \Sigma V^T,
\]
with $U \in \mathbb{R}^{n \times p}$, $\Sigma \in \mathbb{R}^{p \times p}$, and $V \in \mathbb{R}^{p \times p}$.\end{varwidth}
\State $Z \pluseq \phi_i(x) U \mathrm{arctan}(\Sigma) V^T$
\EndFor
\State Compute the thin singular value decomposition
\[
Z = U \Theta V^T,
\]
with $U \in \mathbb{R}^{n \times p}$, $\Sigma \in \mathbb{R}^{p \times p}$, and $V \in \mathbb{R}^{p \times p}$.
\State \label{alg_grass_A} $A \gets \bar{A}_1 V \cos(\Theta) + U \sin(\Theta)$
\State\Return $\mathrm{span}(a_1,a_2,\dots,a_p)$, where $a_j$ denotes the $j^{th}$ column of $A$.
\end{algorithmic}
\end{algorithm} 

Note that the output of Algorithm~\ref{alg_grass} is independent of the choice of orthonormal basis made for each $\mathcal{V}^{(i)}$ in Line~\ref{alg_grass_basis} of Algorithm~\ref{alg_grass}.  This can be checked directly by observing that a change of basis corresponds to post-multiplication of $A_1$ by a matrix $R \in O(p)$, leaving $(I-\bar{A}_1 \bar{A}_1^T) A_1 (\bar{A}_1^T A_1)^{-1}$ invariant.  Similarly, the output of the algorithm is invariant under post-multiplication of $\bar{A}_1$ by any matrix $R \in O(p)$, since it can be checked that such a transformation changes the output of Line~\ref{alg_grass_A} from $A$ to $AR$, whose columns span the same space as those of $A$.  This last statement leads to the conclusion that
\begin{equation} \label{interp_grass_basis_change}
\mathcal{I}_{\bar{A}Q} \tilde{\mathcal{V}}(x) = \mathcal{I}_{\bar{A}} \tilde{\mathcal{V}}(x) 
\end{equation}
for any $Q \in O(p) \times O(n-p)$, which reaffirms~(\ref{postmult}).

The interpolant so constructed enjoys the following additional property.

\begin{lemma}
The interpolant~(\ref{interp_grass}) commutes with the action of $O(n)$ on $Gr(p,n)$.  That is, if $Q \in O(n)$ and $\tilde{\mathcal{V}}^{(i)} = Q\mathcal{V}^{(i)}$, $i=1,2,\dots,m$, then
\[
\mathcal{I}_{Q\bar{A}} \tilde{\mathcal{V}}(x) = Q\mathcal{I}_{\bar{A}} \mathcal{V}(x)
\]
for every $x \in \Omega$.
\end{lemma}
\begin{proof}
Apply~(\ref{general_equivariance}).
\end{proof}
  
Another $O(n)$-equivariant interpolant on $Gr(p,n)$ is given abstractly by~(\ref{Riemannianmean2}).  In this setting, this interpolant is obtained by solving
\[
\sum_{i=1}^m \phi_i(x) F_{\bar{A}}^{-1}(\mathcal{V}^{(i)}) = 0
\]
for $\bar{A}$ and outputting the space spanned by the first $p$ columns of $\bar{A}$.  Algorithmically, this amounts to wrapping a fixed point iteration around Algorithm~\ref{alg_grass}, as detailed below.

\begin{algorithm}[H] 
\caption{Iterative interpolation on the Grassmannian $Gr(p,n)$} \label{alg_grass_Riemannian}
\begin{algorithmic}[1]
\Require Subspaces $\{\mathcal{V}^{(i)} \in Gr(p,n)\}_{i=1}^m$, shape functions $\{\phi_i : \Omega \rightarrow \mathbb{R} \}_{i=1}^m$, point $x \in \Omega$, matrix~$\bar{A}_1 \in \mathbb{R}^{n \times p}$ with orthonormal columns
\Repeat
\State \begin{varwidth}[t]{\linewidth} Use Algorithm~\ref{alg_grass} to compute the interpolant of $\{\mathcal{V}^{(i)}\}_{i=1}^m$ at $x$, storing the result as a matrix $A \in \mathbb{R}^{n \times p}$ (i.e., the matrix $A$ appearing in line~\ref{alg_grass_A} of Algorithm~\ref{alg_grass}).
\end{varwidth}
\State $\bar{A}_1 \gets A$
\Until{converged}
\State\Return $\mathrm{span}(a_1,a_2,\dots,a_p)$, where $a_j$ denotes the $j^{th}$ column of $\bar{A}_1$.
\end{algorithmic}
\end{algorithm} 

Since $O(n)$ is compact, Lemma~\ref{lemma:Riemannian_mean} shows that Algorithm~\ref{alg_grass_Riemannian} produces the weighted Riemannian mean on $Gr(p,n)$.  This interpolant has been considered previously by several authors, including~\cite{begelfor2006affine,chang2012feature,grohs2013quasi}.

\section{Conclusion}

This paper has presented a family of structure-preserving interpolation operators for functions taking values in a symmetric space $\mathcal{S}$.  We accomplished this by identifying $\mathcal{S}$ with a homogeneous space $G/G^\sigma$ and interpolating coset representatives obtained from the generalized polar decomposition.  The resulting interpolation operators enjoy equivariance with respect to the action of $G^\sigma$ on $\mathcal{S}$, as well as equivariance with respect to the action of certain geodesic symmetries on $\mathcal{S}$.  Numerical evidence in Section~\ref{sec:lorentz} suggests that these interpolation operators also enjoy optimal approximation properties, but further work is needed to confirm this theoretically.  In certain cases, namely those addressed in Lemma~\ref{lemma:Riemannian_mean}, the work of~\cite{grohs2015optimal,grohs2013quasi} supplies the needed theoretical confirmation.  Presuming that similar results hold more generally, the application of these interpolation schemes seems intriguing, particularly in the context of numerical relativity, where they provide structure-preserving finite elements for the metric tensor.

\section{Acknowledgements}

EG has been supported in part by NSF under grants DMS-1411792, DMS-1345013. ML has been supported in part by NSF under grants DMS-1010687, CMMI-1029445, DMS-1065972, CMMI-1334759, DMS-1411792, DMS-1345013.

\begin{appendices}

\section{Second-Order Derivatives of the Matrix Exponential} \label{app:d2exp}

In this section, we prove~(\ref{d2Iu}) by showing that if $P : \Omega \rightarrow \mathbb{R}^{n \times n}$ is a smooth matrix-valued function defined on a domain $\Omega \subset \mathbb{R}^d$, then, for each $j,k=1,2,\dots,d$, the matrix $\frac{\partial^2}{\partial x_j \partial x_k} \exp(P(x))$ is given by reading off the $(1,4)$ block of 
\[
\mathrm{exp}\begin{pmatrix} X & Y & Z & W \\ 0 & X & 0 & Z \\ 0 & 0 & X & Y \\ 0 & 0 & 0 & X \end{pmatrix},
\]
where $X = P(x)$, $Y = \frac{\partial P}{\partial x_j}(x)$, $Z = \frac{\partial P}{\partial x_k}(x)$, and $W = \frac{\partial^2 P}{\partial x_j \partial x_k}(x)$.  To prove this, recall first the identity~(\ref{blockexp}), which can be written as
\begin{equation} \label{blockexp2}
\left.\frac{d}{d t}\right|_{t=0} \exp(U+tV) = \left[ \exp \begin{pmatrix} U & V \\ 0 & U \end{pmatrix} \right]_{(1,2)}
\end{equation}
for any square matrices $U$ and $V$ of equal size, where $B_{(1,2)}$ denotes the $(1,2)$ block of a block matrix $B$.  Now observe that with $X$, $Y$, $Z$, and $W$ defined as above,
\begin{align*}
\frac{\partial^2}{\partial x_j \partial x_k} \exp(P(x)) 
&= \left.\frac{\partial^2}{\partial s \, \partial t}\right|_{s=t=0} \exp(X+tY+sZ+stW) \\
&= \left.\frac{\partial}{\partial s}\right|_{s=0} \left.\frac{\partial}{\partial t}\right|_{t=0}  \exp(X+sZ+t(Y+sW)) \\
&= \left.\frac{\partial}{\partial s}\right|_{s=0} \left[\exp\begin{pmatrix} X+sZ & Y+sW \\ 0 & X+sZ \end{pmatrix} \right]_{(1,2)} \\
&= \left[ \left.\frac{\partial}{\partial s}\right|_{s=0} \exp\begin{pmatrix} X+sZ & Y+sW \\ 0 & X+sZ \end{pmatrix} \right]_{(1,2)} 
\end{align*}
Using~(\ref{blockexp2}) again, we have
\[
\left.\frac{\partial}{\partial s}\right|_{s=0} \exp\begin{pmatrix} X+sZ & Y+sW \\ 0 & X+sZ \end{pmatrix} = \left[\exp\begin{pmatrix} X & Y & Z & W \\ 0 & X & 0 & Z \\ 0 & 0 & X & Y \\ 0 & 0 & 0 & X \end{pmatrix}\right]_{(1,2)},
\]
which shows that
\[
\frac{\partial^2}{\partial x_j \partial x_k} \exp(P(x)) = \left[ \left[\exp\begin{pmatrix} X & Y & Z & W \\ 0 & X & 0 & Z \\ 0 & 0 & X & Y \\ 0 & 0 & 0 & X \end{pmatrix}\right]_{(1,2)} \right]_{(1,2)} = \left[\exp\begin{pmatrix} X & Y & Z & W \\ 0 & X & 0 & Z \\ 0 & 0 & X & Y \\ 0 & 0 & 0 & X \end{pmatrix}\right]_{(1,4)}.
\]

\end{appendices}

\printbibliography

\end{document}